\documentclass[letterpaper,10pt,conference]{ieeeconf} 

\IEEEoverridecommandlockouts 

\usepackage{amsmath,amssymb,amsthm,mathtools}
\usepackage[dvipsnames]{xcolor}
\usepackage{graphicx}
\usepackage{booktabs}
\usepackage{array}
\usepackage[compress]{cite}
\usepackage[multiple]{footmisc}

\makeatletter
\newtheoremstyle{stylename}
  {\z@}
  {\z@}
  {}
  {\parindent}
  {\itshape}
  {:}
  {5\p@ plus\p@ minus\p@\relax}
  {\thmname{#1}\thmnumber{ #2}\thmnote{ #3}}
\def\QED{\mbox{\rule[0pt]{1.3ex}{1.3ex}}}

\newlength{\myparindent}
\makeatother
\theoremstyle{stylename}

\newtheorem{theorem}{Theorem}

\newtheorem{lemma}[theorem]{Lemma}
\newtheorem{proposition}[theorem]{Proposition}

\newtheorem{definition}{Definition}

\newtheorem{remark}{Remark}

\newcommand{\calG}{\mathcal{G}}

\newcommand{\calN}{\mathcal{N}}

\newcommand{\calQ}{\mathcal{Q}}

\newcommand{\calS}{\mathcal{S}}

\newcommand{\calU}{\mathcal{U}}
\newcommand{\calV}{\mathcal{V}}
\newcommand{\calW}{\mathcal{W}}
\newcommand{\calX}{\mathcal{X}}

\newcommand{\Ahat}{\hat{A}}

\newcommand{\xhat}{\hat{x}}

\newcommand{\epsilonhat}{\hat\epsilon}
\newcommand{\Sigmahat}{\hat\Sigma}

\newcommand{\sbar}{\bar{s}}

\newcommand{\epsilonbar}{\bar{\epsilon}}
\newcommand{\sigmabar}{\bar{\sigma}}
\newcommand{\mubar}{\bar{\mu}}

\newcommand{\Itilde}{\tilde{I}}

\newcommand{\Ptilde}{\tilde{P}}

\newcommand{\Sigmatilde}{\tilde{\Sigma}}

\renewcommand{\Re}{\mathbb{R}}

\newcommand{\Ne}{\mathbb{N}}

\newcommand{\nxn}{{n\times n}}

\newcommand{\diff}{\mathrm{d}}

\newcommand{\conv}{\mathrm{conv}}

\newcommand{\diag}{\mathrm{diag}}
\newcommand{\trace}{\mathrm{trace}}

\newcommand{\Ker}{\mathrm{Ker}}

\newcommand{\spann}{\mathrm{span}}

\newcommand{\Imag}{\mathrm{Im}}
\newcommand{\Prob}{\mathbb{P}}
\newcommand{\SSb}{\mathbb{S}}
\newcommand{\BBb}{\mathbb{B}}
\newcommand{\sign}{\mathrm{sign}}

\newcommand{\calXhat}{\hat\calX}

\newcommand{\smalltxt}{\mathrm{small}}
\newcommand{\largetxt}{\mathrm{large}}

\title{\LARGE\bfseries Data-driven invariant subspace identification for black-box switched linear systems}

\author{Guillaume O.\ Berger, Raphaël M.\ Jungers and Zheming Wang
\thanks{GB is a BAEF fellow.
He is with CUPLV lab, University of Colorado Boulder.
guillaume.berger@colorado.edu.
RJ is a FNRS honorary Research Associate.
This project has received funding from the European Research Council under the European Union's Horizon 2020 research and innovation program under grant agreement No.\ 864017 -- L2C.
RJ is also supported by the Walloon Region, the Innoviris Foundation, and the FNRS (Chist-Era Druid-net).
RJ and ZW are with ICTEAM institute, UCLouvain, Belgium.
\{raphael.jungers,zheming.wang\}@uclouvain.be.}}

\newif\ifextended
\extendedfalse
\extendedtrue

\begin{document}

\maketitle
\thispagestyle{empty}
\pagestyle{empty}

\begin{abstract}
We present an algorithmic framework for the identification of candidate invariant subspaces for switched linear systems.
Namely, the framework allows to compute an orthonormal basis in which the matrices of the system are close to block-triangular matrices, based on a finite set of observed one-step trajectories and with a priori confidence level.
The link between the existence of an invariant subspace and a common block-triangularization of the system matrices is well known.
Under some assumptions on the system, one can also infer the existence of an invariant subspace when the matrices are close to be block-triangular.
Our approach relies on quadratic Lyapunov analysis and recent tools in scenario optimization.
We present two applications of our results for problems of consensus and opinion dynamics; the first one allows to identify the disconnected components in a switching hidden network, while the second one identifies the stationary opinion vector of a switching gossip process with antagonistic interactions.
\end{abstract}

\section{Introduction}\label{sec-introduction}

Data-driven control systems have a rich and long history, encompassing system identification \cite{lauer2019hybrid}, controller synthesis \cite{hjalmarsson1998iterative,campi2002virtual,coulson2019dataenabled}, formal verification \cite{fan2017dryvr}, etc.
In recent years, we have seen a paradigm shift from \emph{learning models} (system identification) to \emph{learning solutions} (or certificates, like control barrier functions \cite{robey2021learning}, abstractions \cite{makdesi2021efficient}, etc.).
This has been enabled by the development of new techniques for learning, such as PAC (Probably Approximately Correct) Learning \cite{kearns1994anintroduction,shalevshwartz2014understanding} and scenario optimization \cite{margellos2014ontheroad,garatti2021therisk}, which allow to provide solutions that are satisfactory with a high confidence level.
This is particularly relevant for cyber-physical systems because system identification is in general very hard for these systems \cite{lauer2019hybrid}, and so are most control problems, even if the model of the system is known \cite{blondel1999complexity}.

In this paper, we study the problem of identifying invariant subspaces for black-box switched linear systems in a data-driven way.
Switched linear systems are systems described by a finite set of linear modes among which the system can switch over time.
They constitute a paradigmatic class of cyber-physical systems, and appear naturally in many engineering applications or as abstractions of more complex systems \cite{liberzon2003switching}.
Invariant subspaces are a central concept in linear system analysis; e.g., for safety verification, to perform dimensionality reduction in system analysis, or because they contain important information about the system (as in consensus \cite{blondel2005convergence} or Markov chains \cite{seneta1981nonnegative}).

Our approach for subspace identification draws on scenario optimization \cite{garatti2021therisk}.
In particular, the existence of a candidate invariant subspace is inferred from the existence of a degenerate Lyapunov function for the system.
More precisely, the zero level-set of the Lyapunov function provides an orthonormal basis in which the matrices of the system are close (with bounded distance) to be block-triangular.
Under some assumptions on the system, one can then infer the existence of an invariant subspace for the system; examples of such assumptions are discussed in the applications.
By restricting to \emph{quadratic} Lyapunov functions, the existence of such a function can be formulated as an SDP optimization problem.
Because the system is black-box, the function has to be computed using a finite amount of data, and thus it is not guaranteed that the obtained quadratic function is a valid Lyapunov function for the system.
However, by using results from scenario optimization, we can estimate with a priori-fixed confidence level, the probability of the set of one-step trajectories that are compatible with the obtained function.
Using this estimation, we can bound the distance of the system matrices (in the appropriate orthonormal basis) to the set of block-triangular matrices. 
Note that a similar approach has been used in \cite{kenanian2019data,berger2021chanceconstrained,wang2021datadriven,berger2022comments} for the data-driven computation of quadratic Lyapunov functions for stability analysis of switched linear systems.
However, to the best of our knowledge, this work is the first one addressing the problem of common triangularization and invariant subspace identification for switched linear systems, without passing through a system identification phase.

We apply our technique on two problems in opinion dynamics.
The first one is a problem of consensus over a switching hidden network \cite{jadbabaie2003coordination}: by estimating the dimension of the dominant invariant subspace of the system, we are able to infer, from a finite set of observed one-step trajectories, the number of disconnected components in the network.
The second application is a problem of opinion dynamics with antagonistic interactions \cite{altafini2013consensus}.
The attracting opinion vector, if it exists, corresponds to a $1$-dimensional invariant subspace of the system.
Our technique allows us to identify, in a data-driven way, with high confidence level, such a subspace.


\emph{Outline.}
The paper is organized as follows.
In Section~\ref{sec-problem-setting}, we introduce the problem of interest.
Several intermediate results follow, leading eventually to the main result of the paper in Subsection~\ref{ssec-main-result}, which allows to identify an orthonormal basis in which the system matrices are close to be block-triangular.
Finally, in Section~\ref{sec-application-consensus-opinion-dynamics}, we present two applications of our framework for problems of consensus and opinion dynamics.

\emph{Notation.}
For vectors, $\lVert\cdot\rVert$ denotes the Euclidean norm, and for matrices, it denotes the spectral norm (largest singular value).
$\lVert\cdot\rVert_F$ denotes the Frobenius norm.
For a set $\calX\subseteq\Re^d$, $\calX^\perp$ denotes its orthogonal complement and $\conv(\calX)$ its convex hull.
$\SSb^{d-1}$ is the unit Euclidean sphere in $\Re^d$.

\section{Problem setting and main result}\label{sec-problem-setting}

\subsection{Switched linear systems and invariant subspaces}

We consider a discrete-time switched linear system, described by
\begin{equation}\label{eq-switched-system}
\xi(t+1) = A_{\sigma(t)} \xi(t), \quad \xi(t)\in\Re^n, \quad t\in\Ne,
\end{equation}
where $\sigma:\Ne\to\calQ\coloneqq\{1,\ldots,Q\}$ and for each $q\in\calQ$, $A_q\in\Re^\nxn$.
The function $\sigma$ is called the \emph{switching signal}\footnote{In our framework (worst-case scenario analysis), the switching signal is an external input on which the user has no control, and the objective is to deduce properties of the systems that will be valid for every switching signal.} of the system and specifies which \emph{mode} (i.e., which transition matrix $A_q$) is used by the system at each time $t\in\Ne$.

An invariant subspace for System \eqref{eq-switched-system} is a (non-trivial) linear subspace $\calU\subseteq\Re^n$ satisfying that for every trajectory $\xi$ of \eqref{eq-switched-system} with $\xi(0)\in\calU$, it holds that for all $t\in\Ne$, $\xi(t)\in\calU$.
The existence of an invariant subspace is related to the common block-triangularization of the system matrices, as explained in the proposition below.

\begin{proposition}[{\cite[p.\ 12]{jungers2009thejoint}}]\label{pro-triangularized}
Let $\calU\subseteq\Re^n$ be a non-trivial linear subspace with dimension $r$, and let $U\in\Re^\nxn$ be an orthogonal matrix whose first $r$ columns are a basis of $\calU$.
The following are equivalent:
\begin{enumerate}
\item $\calU$ is invariant for System \eqref{eq-switched-system};
\item For each $q\in\calQ$, there is $A_q^{(11)}\in\Re^{r\times r}$, $A_q^{(12)}\in\Re^{r\times(n-r)}$ and $A_q^{(22)}\in\Re^{(n-r)\times(n-r)}$ such that
\begin{equation}\label{eq-triangularizable}
U^\top A_qU = \left[\begin{array}{cc} A_q^{(11)} & A_q^{(12)} \\ 0 & A_q^{(22)} \end{array}\right].
\end{equation}\vskip0pt
\end{enumerate}
\end{proposition}

Finding an invariant subspace from a finite set of trajectories is generally impossible if none of the trajectories starts inside the subspace.
Therefore, we focus on finding a common \emph{``near''} block-triangularization, that is, an orthogonal change of basis $U$ in which the norm of the lower-left blocks --- the would-be $A_q^{(21)}$ blocks in \eqref{eq-triangularizable} --- is bounded.

In our framework, the available data is a set of $N$ one-step trajectories $(x_i,y_i)$, where $y_i=A_{q_i}x_i$ for some \emph{unobserved} mode $q_i$; i.e., we observe $N$ starting points $x_i$ and their respective images $y_i$ by some latent mode $q_i$.

\subsection{Quadratic Lyapunov approach}

A common block-triangularization of System \eqref{eq-switched-system} can be obtained from a positive semi-definite (PSD) matrix $P$ satisfying that for all $q\in\calQ$, $A_q^\top PA_q^{}\preceq\gamma^2P$ for some $\gamma>0$.
Indeed, in that case, the kernel $\calU$ of $P$ gives an invariant subspace.%
\footnote{This is easily seen, as for any $x\in\Ker(P)$, $(A_qx)^\top P(A_qx)\leq\gamma^2x^\top Px=0$, so that $PA_qx=0$ and thus $A_qx\in\Ker(P)$.}
Note that $\gamma$ can be seen as a Lagrange multiplier arising in the $\calS$-procedure \cite{boyd1994linear}.

When the matrices $A_q$ are not available, we consider the data-driven version of the above approach.
Namely, for a set $\Omega=\{(x_i,y_i)\}_{i=1}^N$ of observations and $\gamma>0$, we consider the following optimization problem:
\begin{equation}\label{eq-optim-black}
\begin{array}{r@{}l}
\min\limits_P&\quad \lVert P\rVert_F^2 \\
\text{s.t.}&\quad y_i\!^\top Py_i \leq \gamma^2\,x_i\!^\top Px_i, \quad \forall\,i\in\{1,\ldots,N\}, \\
&\quad P\succeq0,\: \trace(P)=1.
\end{array}
\end{equation}
We let $P^\star_{\Omega,\gamma}$ be the optimal solution of \eqref{eq-optim-black}, if it exists.
In the following, we will assume that $\gamma>0$ is fixed, and thus we will omit it in the notation.

The rationale of minimizing $\lVert P\rVert_F^2$ (which is equal to the sum of the squared eigenvalues of $P$) is to reduce the gap between the largest and smallest nonzero eigenvalues of $P^\star_\Omega$.
In particular, if one defines the skewness of a PSD matrix $P$ as in Definition~\ref{def-skewness} below, then minimizing $\lVert P\rVert_F$ induce its skewness to be small, which will be desirable in Theorems~\ref{thm-probability-homogeneity} and~\ref{thm-main-result}.

\begin{definition}\label{def-skewness}
Let $P\in\Re^\nxn$ be PSD with eigenvalues $\lambda_1\geq\ldots\geq\lambda_{n-r}>0=\lambda_{n-r+1}=\ldots=\lambda_n$.
We define the \emph{skewness} of $P$ by $\chi(P)=\lambda_{n-r}^{-1}\prod_{j=1}^{n-r}\lambda_j^{1/n}$.
\end{definition}

We now establish the link between a solution of \eqref{eq-optim-black} and the existence of an orthogonal change of basis allowing to near block-triangularize the matrices of System \eqref{eq-switched-system}.
Therefore, define the set
\begin{align}
&\Psi_\Omega= \{(x,q)\in\Re^n\times\calQ : \nonumber \\
&\hspace{2.6cm} (A_qx)^\top P^\star_\Omega(A_qx) \leq \gamma^2 x^\top P^\star_\Omega x\} \label{eq-omegabar}
\end{align}
of point--mode pairs compatible with $P^\star_\Omega$.

It holds that if $\Psi_\Omega$ is sufficiently covering the set $\Re^n\times\calQ$, then one can bound the norm of the lower-left blocks in a decomposition of System \eqref{eq-switched-system} akin to \eqref{eq-triangularizable}, where $U$ is given by the kernel of $P^\star_\Omega$.
This is stated in Theorem~\ref{thm-near-triangularizable} below, but first we formalize the notion of ``sufficiently covering'' with the following concept of homogeneity (see also Figure~\ref{fig-homogeneous} for an illustration).

\begin{definition}\label{def-homogeneous}
Let $\calX\subseteq\Re^n$, $\epsilon>0$ and $P_1,P_2\in\Re^\nxn$ be PSD.
We say that $\calX$ is \emph{$(\epsilon,P_1,P_2)$-homogeneous} if for every $x\in\Re^n$, there is $\calV\subseteq\Re^n$ such that $0\in\conv(\calV)$ and for every $v\in\calV$, $v\!^\top P_1x=0$, $v\!^\top P_1v\leq\epsilon^2x^\top P_2x$ and $x+v\in\calX$.
\end{definition}

\begin{figure}
\centering
\includegraphics[width=0.5\linewidth]{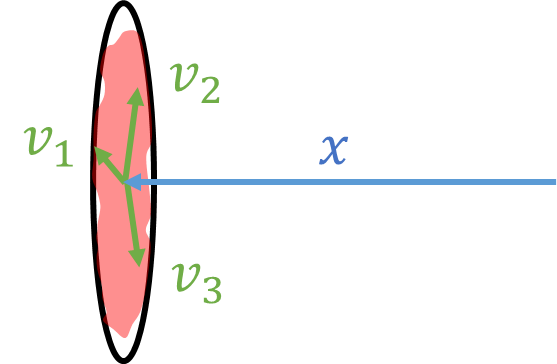}
\caption{$(\epsilon,P_1,P_2)$-homogeneity.
The red region is the intersection of $\calX$ with $x+\{v\in\Re^n: v^\top P_1x=0,\,v\!^\top P_1v\leq\epsilon^2x^\top P_2x\}$.
$\calX$ is $(\epsilon,P_1,P_2)$-homogeneous if for every $x$, there are vectors (here, e.g., $v_1,v_2,v_3$) whose convex hull contains $x$.}
\label{fig-homogeneous}
\end{figure}

\begin{theorem}\label{thm-near-triangularizable}
Consider System \eqref{eq-switched-system}.
Let $\Omega\subseteq\Re^n\times\Re^n$ and $\epsilon\in(0,1)$.
Let $P^\star_\Omega$ be the optimal solution of \eqref{eq-optim-black} and let $\Psi_\Omega$ be as in \eqref{eq-omegabar}.
Let $P^\star_\Omega=U\Sigma^2U^\top$ with $U\in\Re^\nxn$ orthogonal, $\Sigma=\diag(0_1,\ldots,0_r,\lambda_1,\ldots,\lambda_{n-r})$ and $\lambda_1,\ldots,\lambda_{n-r}>0$.
Let $\Sigmatilde=\diag(1_1,\ldots,1_r,\lambda_1,\ldots,\lambda_{n-r})$ and $\Ptilde=U\Sigmatilde^2U^\top$.
For each $q\in\calQ$, let $\calX_q=\{x:(x,q)\in\Psi_\Omega\}$ and assume that $\calX_q$ is $(\epsilon,P^\star_\Omega,\Ptilde)$-homogeneous.
Then, for each $q\in\calQ$, there is $A_q^{(11)}\in\Re^{r\times r}$, $A_q^{(12)}\in\Re^{r\times(n-r)}$, $A_q^{(21)}\in\Re^{(n-r)\times r}$ and $A_q^{(22)}\in\Re^{(n-r)\times(n-r)}$ such that
\begin{equation}\label{eq-near-triangularizable}
U^\top A_qU = \left[\begin{array}{cc} A_q^{(11)} & A_q^{(12)} \\ A_q^{(21)} & A_q^{(22)} \end{array}\right],
\end{equation}
and $\lVert\Lambda A_q^{(21)}\rVert\leq\epsilon\gamma$ and $\lVert\Lambda A_q^{(22)}\Lambda^{-1}\rVert\leq\sqrt{1+\epsilon^2}\gamma$, where $\Lambda=\diag(\lambda_1,\ldots,\lambda_{n-r})$.
\end{theorem}

\begin{proof}
\ifextended%
See Appendix~\ref{ssec-app-proof-thm-near-trianularizable}.
\fi
\end{proof}

\subsection{Scenario approach and probabilistic guarantees}\label{ssec-scenario-approach-probabilistic-guarantees}

Theorem~\ref{thm-near-triangularizable} states that if the set $\Psi_\Omega$ satisfies some homogeneity assumption, then we can obtain a common near block-triangularization of System \eqref{eq-switched-system} with a bound on the norm of the lower-left blocks.
The difficulty is that it is in general impossible to compute the set $\Psi_\Omega$ from $\Omega=\{(x_i,y_i)\}_{i=1}^N$ if we do not have access to the matrices of the system.
Nevertheless, if the point--mode pairs $(x_i,q_i)$ generating $\Omega$ are sampled independently at random, then we can obtain \emph{probabilistic guarantees} on the probability of $\Psi_\Omega$.
This is presented in the following theorem, which is the first main result of this subsection.

\begin{theorem}\label{thm-chance-constrained-probabilistic-bound-optim-black}
Consider System \eqref{eq-switched-system} and let $\Prob$ be a probability measure on $\Re^n\times\calQ$.
Let $\beta\in(0,1)$ be a confidence level and $N\in\Ne$ a number of samples.
Let $\Omega=\{(x_i,y_i)\}_{i=1}^N$, where $y_i=A_{q_i}x_i$ and $\{(x_i,q_i)\}_{i=1}^N$ is sampled i.i.d.\ according to $\Prob$.
Let $P^\star_\Omega$ be the optimal solution of \eqref{eq-optim-black} and $\Psi_\Omega$ be as in \eqref{eq-omegabar}.
Then, with probability $1-\beta$ on the sampling of $\{(x_i,q_i)\}_{i=1}^N$, it holds that $\Prob(\Psi_\Omega)\geq1-\epsilonbar(\frac{n(n+1)}2)$, where $\epsilonbar:\{0,\ldots,N\}\to[0,1]$ is defined by $\epsilonbar(N)=1$ and for all $k\in\{0,\ldots,N-1\}$, $\epsilonbar(k)$ is the solution of the equation
\begin{equation}\label{eq-beta-N-eps}
\binom{N}{k}(1-\epsilonbar(k))^{N-k} - \frac\beta{N}\sum_{i=k}^{N-1}\binom{i}{k}(1-\epsilonbar(k))^{i-k}=0.
\end{equation}
\end{theorem}

\begin{proof}
\ifextended%
The proof uses tools from scenario optimization \cite{garatti2021therisk} and convex optimization \cite{calafiore2010random}.
The proof is presented in Appendix~\ref{ssec-app-proof-thm-chance-constrained-probabilistic-bound-optim-black}.
\fi
\end{proof}


\ifextended%
The dependence of $\epsilonbar(\frac{n(n+1)}2)$ on $N$ for several values of $\beta$ and $n$ is depicted in Figure \ref{fig-epsilonbar}.
We see that $1/\epsilonbar(\frac{n(n+1)}2)$ grows polynomially with $N$.
\fi

\ifextended
\begin{figure}
\centering
\includegraphics[width=\linewidth]{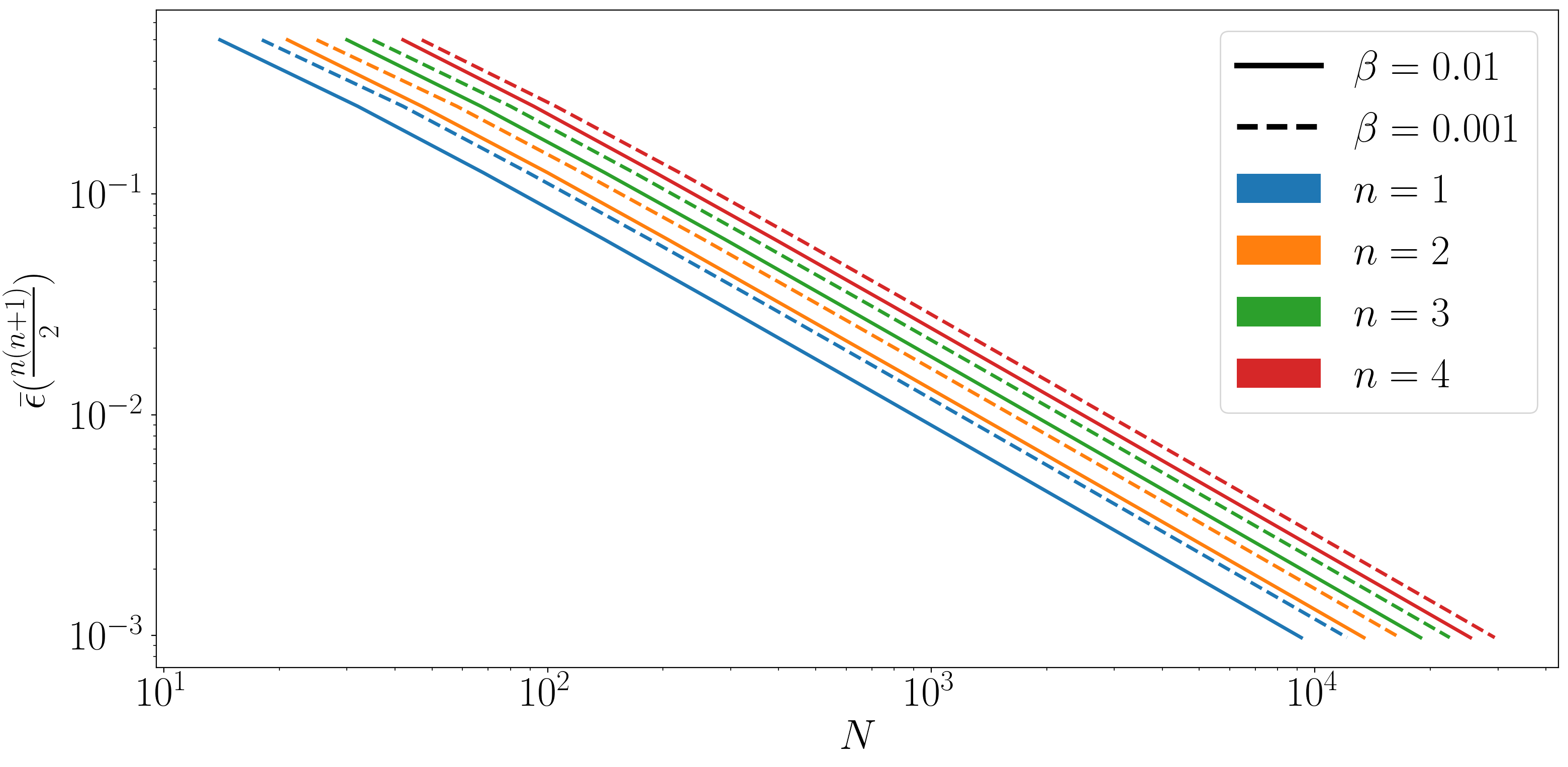}
\caption{Dependence of $\epsilonbar(\frac{n(n+1)}2)$ on $N$ for several values of $\beta$ and $n$.}
\label{fig-epsilonbar}
\end{figure}
\fi

\begin{remark}\label{rem-constraints-P}
In some situations, it is useful to constrain the variable $P$ in \eqref{eq-optim-black}; e.g., to include prior information on $P$ or to reduce the computation time by lowering the dimension of the problem.
If the additional constraints are convex, then Theorem \ref{thm-chance-constrained-probabilistic-bound-optim-black} applies in the very same way; the only thing that needs to be changed is to use $\epsilonbar(d+1)$ instead of $\epsilonbar(\frac{n(n+1)}2)$, where $d$ is the number of degrees of freedom of the variable $P$.
Furthermore, if one does not want to fix $\gamma$ a priori, but rather wants to find the smallest $\gamma$ for which \eqref{eq-optim-black} is feasible (e.g., to make the bound on the norm of the lower blocks in \eqref{eq-near-triangularizable} as small as possible), then the problem becomes quasi-convex and one has to use $\epsilonbar(2d+1)$ instead of $\epsilonbar(d+1)$;
\ifextended%
see, e.g., \cite[Theorem~3.1]{eppstein2005quasiconvex}, which provides the bound $s^*(\psi)\leq 2d+1$ in Theorem~\ref{thm-basis-cardinality} for quasi-convex problems.
\fi
This approach is used for instance in the applications in Section \ref{sec-application-consensus-opinion-dynamics}.
\end{remark}

Theorem~\ref{thm-chance-constrained-probabilistic-bound-optim-black} states that with enough observations ($N$ large), we can assume with high confidence that $\Prob(\Psi_\Omega)$ is close to $1$.
Now, we show that we can ensure homogeneity of the components of $\Psi_\Omega$ for each $q\in\calQ$ from this property.
Therefore, we assume in the following that $\Prob$ is the \emph{uniform spherical probability distribution} on $\Re^n\times\calQ$, denoted by $\Prob_\circ$ and defined by:
$\Prob_\circ(\calX\times\{q\})=\mu^{n-1}(\calX\cap\SSb^{n-1})/Q$, where $\mu^{n-1}$ is the uniform spherical measure on $\SSb^{n-1}$
\ifextended%
(see Appendix~\ref{ssec-app-spherical-measure})
\fi
and $Q=\lvert\calQ\rvert$.


\begin{theorem}\label{thm-probability-homogeneity}
Consider System \eqref{eq-switched-system}.
Let $\Omega\subseteq\Re^n\times\Re^n$ and $\epsilonhat\in(0,1)$.
Let $P^\star_\Omega$ be the optimal solution of \eqref{eq-optim-black} and $\Psi_\Omega$ be as in \eqref{eq-omegabar}.
Assume that $\Prob_\circ(\Psi_\Omega)\geq1-\epsilonhat$.
Let $P^\star_\Omega=U\Sigma^2U^\top$, $\Sigmatilde$ and $\Ptilde$ be as in Theorem \ref{thm-near-triangularizable}.
Let $\eta=Q\chi(P^\star_\Omega)\epsilonhat$, where $\chi(P^\star_\Omega)$ is the skewness of $P^\star_\Omega$ (see Definition~\ref{def-skewness}).
Assume that $\eta<\frac12$ and let
\begin{equation}\label{eq-sstar}
\sbar=\sqrt{1-\Itilde_{\frac{n-1}2,\frac12}(2\eta)},
\end{equation}
where $\Itilde_{a,b}$ is the \emph{inverse regularized beta function}.%
\footnote{The inverse regularized beta function is the function $\Itilde_{a,b}:[0,1]\to[0,1]$ defined by $\Itilde_{a,b}(x)=y$ if and only if $\frac{\int_0^y t^{a-1}(1-t)^{b-1}\,\diff t}{\int_0^1 t^{a-1}(1-t)^{b-1}\,\diff t}=x$ \cite{olver2010nist}.}
It holds that for each $q\in\calQ$, $\calX_q$ is $(\epsilon,P^\star_\Omega,\Ptilde)$-homogeneous with $\epsilon=\sqrt{\sbar^{-2}-1}$, where $\calX_q$ is as in Theorem~\ref{thm-near-triangularizable}.
\end{theorem}

\begin{proof}
\ifextended%
See Appendix~\ref{ssec-app-proof-thm-probability-homogeneity}.
\fi
\end{proof}

\ifextended%
The relation \eqref{eq-sstar} between $\sbar$ and $\eta$ is depicted in Figure \ref{fig-sstar}.
\fi

\ifextended
\begin{figure}
\centering
\includegraphics[width=\linewidth]{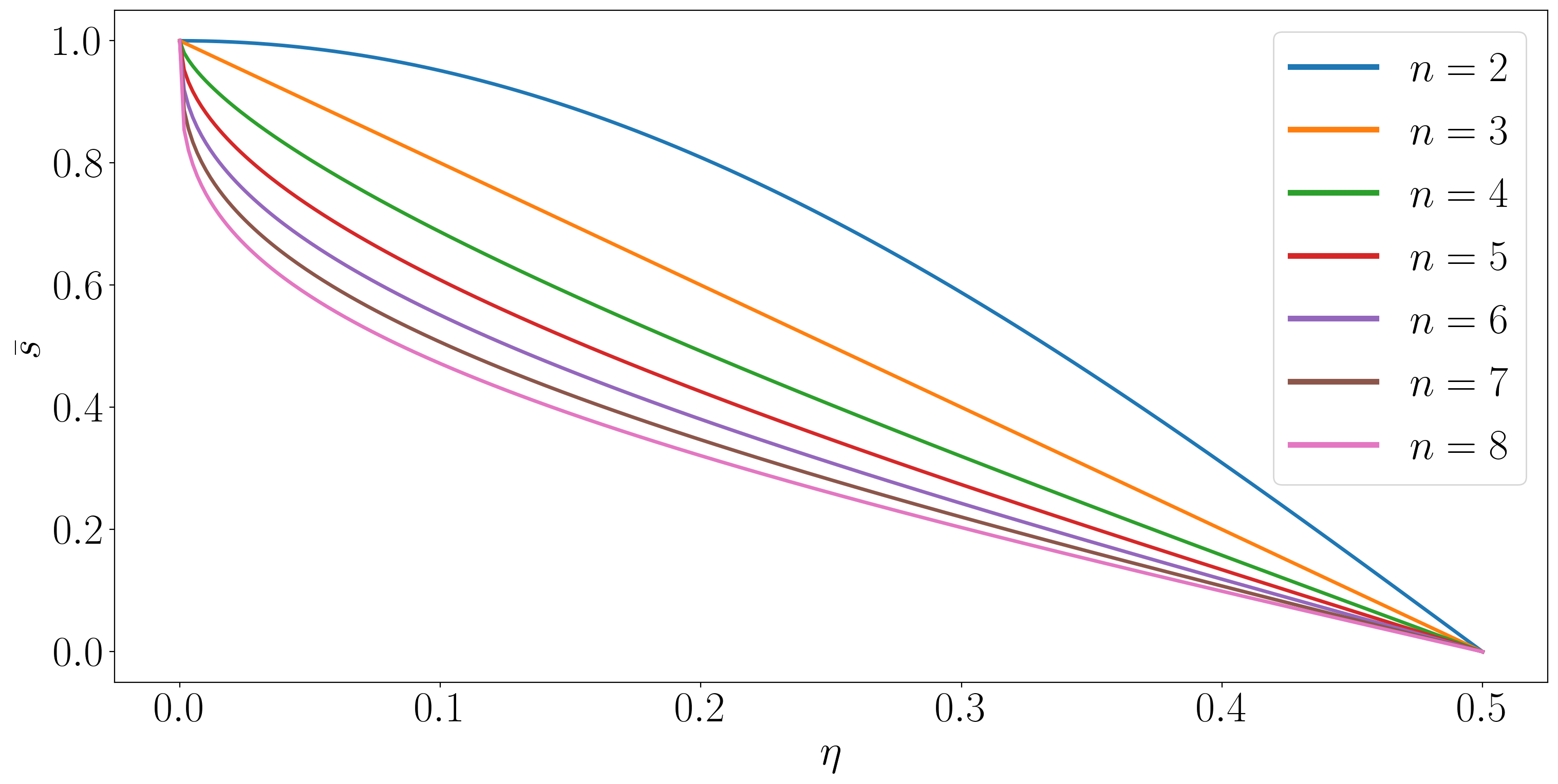}
\caption{$\sbar$ in \eqref{eq-sstar} for several values of $n$.}
\label{fig-sstar}
\end{figure}
\fi

\subsection{Main result}\label{ssec-main-result}

We are now able to state the main result of the paper.

\begin{theorem}[(Main result)]\label{thm-main-result}
Consider System \eqref{eq-switched-system}.
Let $\Omega=\{(x_i,y_i)\}_{i=1}^N$, where $y_i=A_{q_i}x_i$ and $\{(x_i,q_i)\}_{i=1}^N$ is sampled i.i.d.\ according to the uniform spherical probability distribution $\Prob_\circ$.
Let $P^\star_\Omega$ be the optimal solution of \eqref{eq-optim-black}.
Let $P^\star_\Omega=U\Sigma^2U^\top$ with $U\in\Re^\nxn$ orthogonal, $\Sigma=\diag(0_1,\ldots,0_r,\lambda_1,\ldots,\lambda_{n-r})$ and $\lambda_1,\ldots,\lambda_{n-r}>0$.
Let $\beta\in(0,1)$ and $\eta=Q\chi(P^\star_\Omega)\epsilonbar(\frac{n(n+1)}2)$, where $\epsilonbar$ is as in \eqref{eq-beta-N-eps} and $\chi(P^\star_\Omega)$ is the skewness of $P^\star_\Omega$.
Assume that $\eta<\frac12$ and let $\sbar$ be as in \eqref{eq-sstar}.
Then, with probability $1-\beta$ on the sampling of $\{(x_i,q_i)\}_{i=1}^N$, it holds that for each $q\in\calQ$, there is $A_q^{(11)}\in\Re^{r\times r}$, $A_q^{(12)}\in\Re^{r\times(n-r)}$, $A_q^{(21)}\in\Re^{(n-r)\times r}$ and $A_q^{(22)}\in\Re^{(n-r)\times(n-r)}$ such that
\[
U^\top A_qU = \left[\begin{array}{cc} A_q^{(11)} & A_q^{(12)} \\ A_q^{(21)} & A_q^{(22)} \end{array}\right],
\]
and $\lVert\Lambda A_q^{(21)}\rVert\leq\epsilon\gamma$ and $\lVert\Lambda A_q^{(22)}\Lambda^{-1}\rVert\leq\sqrt{1+\epsilon^2}\gamma$, where $\Lambda=\diag(\lambda_1,\ldots,\lambda_{n-r})$ and $\epsilon=\sqrt{\sbar^{-2}-1}$.
\end{theorem}

\begin{proof}
Let $\epsilonhat=\epsilonbar(\frac{n(n+1)}2)$.
By Theorem \ref{thm-chance-constrained-probabilistic-bound-optim-black}, it holds that with probability $1-\beta$ on the sampling, $\Prob_\circ(\Psi_\Omega)\geq1-\epsilonhat$.
Whenever this is the case, it holds by Theorem \ref{thm-probability-homogeneity} that for each $q\in\calQ$, $\calX_q$ is $(\epsilon,P^\star_\Omega,\Ptilde)$-homogeneous, where $\calX_q$ is as in Theorem~\ref{thm-near-triangularizable}.
Hence, we can apply Theorem \ref{thm-near-triangularizable}, concluding the proof.
\end{proof}

Theorem \ref{thm-main-result} says that if we sample enough points ($N$ large), then, unless we are unlucky in our sampling (which happens with probability at most $\beta$), we can find an orthonormal basis in which we can bound the deviation of System \eqref{eq-switched-system} from an exact block-triangular form as in \eqref{eq-triangularizable}.
\ifextended%
A flowchart of the framework and the computation of the different quantities is presented in Figure \ref{fig-flowchart}.
\fi

\ifextended
\begin{figure}
\centering
\includegraphics[width=\linewidth]{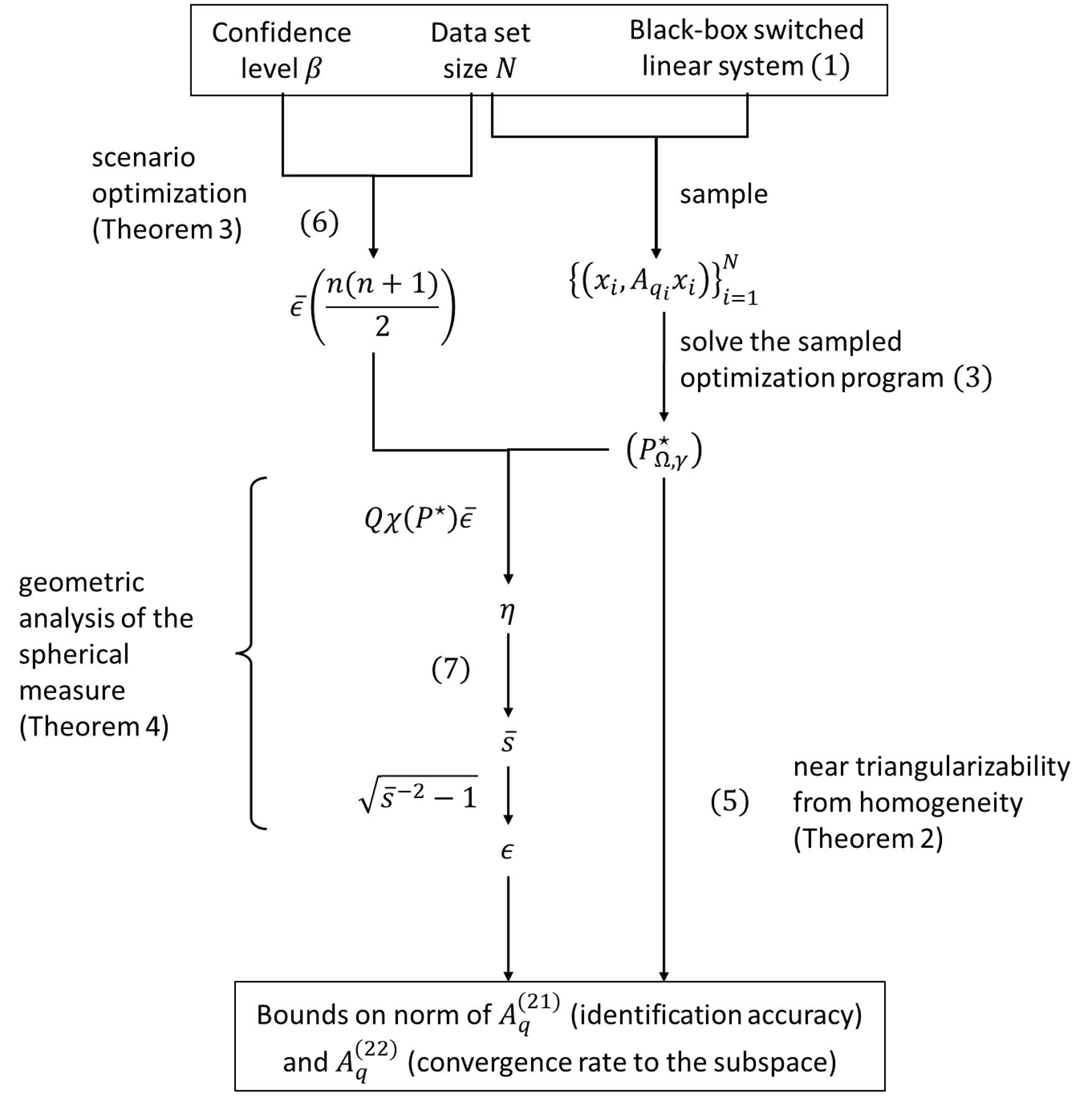}
\caption{Flowchart of the data-driven Lyapunov framework described in Theorem \ref{thm-main-result}.}
\label{fig-flowchart}
\end{figure}
\fi

There are several sources of conservatism in the derivation of Theorem~\ref{thm-main-result}, the main ones are:
\begin{itemize}
\item The condition of homogeneity (Definition \ref{def-homogeneous}) involves every $x\in\Re^n$, while in the proof of Theorem \ref{thm-near-triangularizable}, we need the property of homogeneity to hold only for $x\in\Ker(P^\star_\Omega)\cup\Imag(P^\star_\Omega)$.
The reason we do not refine the definition of homogeneity is that uniform homogeneity is the best that can be deduced from the knowledge that $\Prob_\circ(\Psi_\Omega)\geq1-\epsilonbar$ (Theorem \ref{thm-probability-homogeneity}).
\item In the proof of Theorem \ref{thm-chance-constrained-probabilistic-bound-optim-black}, we use the upper bound $s^*(\psi)\leq d+1$ with $d=\frac{n(n+1)}2-1$, while in practice the value of $s^*(\psi)$ can be much smaller.
However, as noted in \cite[p.\ 609]{garatti2021therisk}, computing the exact value of $s^*(\psi)$ can be difficult as it amounts to look at all possible sub-sequences $\varphi$ of $\psi$, whose number grows combinatorially with $N$.
For this reason, we have considered only the upper bound in the numerical examples.
\item In the proof of Theorem \ref{thm-probability-homogeneity}, it is actually shown that $\calX_q$ is $(\epsilon,\Ptilde,\Ptilde)$-homogeneous, while only the $(\epsilon,P^\star_\Omega,\Ptilde)$-homogeneity (which follows from $P^\star_\Omega\preceq\Ptilde$) is used.
When $P^\star_\Omega$ has many zero eigenvalues, the difference between $P^\star_\Omega$ and $\Ptilde$ can be large.
\ifextended%
The reason we do not refine this result is that the application of Lemma \ref{lem-measure-s-slice-sphere} requires to take slices of a sphere, which can be obtained (after a change of coordinates) only from a non-degenerate ellipsoid.
\fi
\end{itemize}

\section{Application to consensus and opinion dynamics}\label{sec-application-consensus-opinion-dynamics}

The link between the existence of an invariant subspace for the system and a common block-triangularization of the system was established in Proposition~\ref{pro-triangularized}.
However, we have not discussed this link when one has only a common \emph{near} block-triangularization of the system as the one provided in Theorem~\ref{thm-main-result}.
In fact, in general, it is not possible to ensure the existence of an invariant subspace from a non-exact block-triangularization of the system, without further information about the system.
However, there are applications for which the available prior information about the system allows us to do so.
We discuss two such applications in this section.

\subsection{Consensus over switching hidden network}\label{ssec-application-consensus}

We consider the problem of consensus over a switching hidden network.
The evolution of the value of the nodes in the network over time can be modeled as a switched linear system \eqref{eq-switched-system}, where $\xi(t)$ is the state vector ($n$ is the number of nodes) at time $t$ and $A_{\sigma(t)}$ is the \emph{unknown} interaction matrix at time $t$ (see, e.g., \cite{jadbabaie2003coordination}).
For example, let us consider the system described by the networks $\{\calG_q\}_{q\in\calQ}$ depicted in Figure \ref{fig-consensus-graphs} ($n=8$, $Q=3$), and where the value of each node $j\in\{1,\ldots,n\}$ is updated following the rule
\begin{equation}\label{eq-consensus-rule}
\xi^{(j)}(t+1) = \frac1{1+\lvert\calN_j^{\sigma(t)}\rvert}\Big(\xi^{(j)}(t) + \sum_{j'\in\calN_j^{\,\sigma(t)}} \xi^{(j')}(t)\Big),
\end{equation}
with $\calN_j^q$ the set of nodes adjacent to $j$ in $\calG_q$.

\begin{figure}
\centering
\includegraphics[width=0.9\linewidth]{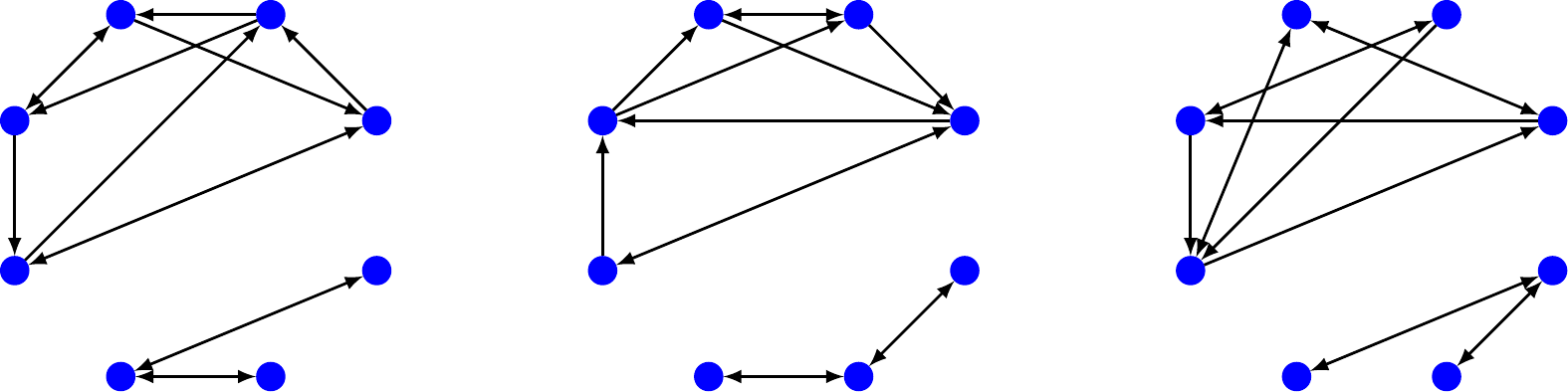}
\caption{Networks of nodes.
A node $j_1$ is \emph{adjacent} to a node $j_2$ if there is an edge from $j_2$ to $j_1$ in the network.}
\label{fig-consensus-graphs}
\end{figure}

In our setting, the networks are not known.
The only things we know are the dimension $n$ and the number of modes $Q$ (or an upper bound on $Q$).
The goal is to identify groups of nodes that \emph{do not} interact with each other, that is, groups of nodes that are \emph{disconnected} in the networks $\{\calG_q\}_{q\in\calQ}$.
To address this problem with an invariant subspace approach, let us introduce the following result.


\begin{proposition}\label{pro-components-invariant-subspace}
Let $\calG$ be a simple directed graph with $n$ nodes, and let $A\in\Re^\nxn$ be the matrix of the associated linear system given by \eqref{eq-consensus-rule}.
Let $\calV\subseteq\{1,\ldots,n\}$ and let $u\in\{0,1\}^n$ be the vector defined by, for all $j\in\{1,\ldots,n\}$, $u^{(j)}=1$ if and only if $j\in\calV$.
Then, $\calV$ and $\{1,\ldots,n\}\setminus\calV$ are disconnected if and only if $Au=u$.
\end{proposition}

\begin{proof}
The ``only if'' direction is clear.
For the ``if'' direction, we proceed by contraposition.
Therefore, assume that there is an edge from a node $j_1\in\calV$ to a node $j_2\notin\calV$ in $\calG$.
Then, $[Au]^{(j_1)}<u^{(j_1)}$, so that $Au\neq u$.
Similarly, if there is an edge from a node $j_1\notin\calV$ to a node $j_2\in\calV$ in $\calG$, then $[Au]^{(j_1)}>u^{(j_1)}$, so that $Au\neq u$.
\end{proof}

To find an invariant subspace of the system given by Figure \ref{fig-consensus-graphs} and \eqref{eq-consensus-rule}, we use the data-driven quadratic Lyapunov framework described in Section \ref{sec-problem-setting}.
First, we compute the optimal solution of \eqref{eq-optim-black} with a relatively small sample set of one-step trajectories $\{(x_i,y_i)\}_{i=1}^{N_\smalltxt}$, $N_\smalltxt=2000$, of the system.
The kernel of the associated matrix $P^\star_\smalltxt$ is given by
\[\Imag\left[\!{
\begin{array}{c@{\;\;}c@{\;\;}c@{\;\;}c@{\;\;}c@{\;\;}c@{\;\;}c@{\;\;}c}
0.26 & 0.25 & 0.25 & 0.22 & 0.29 & -0.48 & -0.49 & -0.46 \\
0.38 & 0.33 & 0.39 & 0.31 & 0.42 & 0.32 & 0.33 & 0.33
\end{array}}\!\right]^\top,
\]
which is close to the subspace $\calU\coloneqq\spann\{u_1,u_2\}$ where $u_1=[1,1,1,1,1,0,0,0]^\top$ and $u_2=[0,0,0,0,0,1,1,1]^\top$.%
\footnote{Here, we used, in a \emph{compressed sensing} fashion, the prior information that the invariant subspace of interest, if it exists, is spanned by binary vectors.}
Hence, we suspect that for all $q\in\calQ$, the connected components of $\calG_q$ have the form $\calV=\{1,\ldots,n\}$, or $\calV\subseteq\{1,\ldots,5\}$, or $\calV\subseteq\{6,\ldots,8\}$.

To verify this hypothesis with high confidence, we solve \eqref{eq-optim-black} with a larger sample set of one-step trajectories ($N_\largetxt=247\,122\,000$) and fixing $P=U^{(2)}(U^{(2)})^\top$ where the columns of $U^{(2)}\in\Re^{n\times(n-2)}$ are an orthonormal basis of $\calU^\perp$.
This time, we also try to minimize $\gamma$ (see also Remark~\ref{rem-constraints-P}).
This provides $\gamma^\star=0.69$.
By fixing the confidence level to $\beta=0.01$, it follows that $\epsilonbar(1)=1.86\cdot10^{-8}$.%
\footnote{We use $\epsilonbar(1)$ because $P$ is fixed, so that $d=0$; see Remark \ref{rem-constraints-P}.}
By Theorem \ref{thm-main-result}, it then follows that the matrices of the system admit a decomposition \eqref{eq-near-triangularizable}, with $\epsilon=0.122$, $\gamma=0.69$, $\Lambda=I_{n-2}$ and $U=\big[\frac{u_1}{\sqrt{5}},\frac{u_2}{\sqrt{3}},U^{(2)}\big]$.

From this decomposition, we can now deduce that for each $q\in\calQ$, the connected components of $\calG_q$ have the form $\calV=\{1,\ldots,n\}$, or $\calV\subseteq\{1,\ldots,5\}$, or $\calV\subseteq\{6,\ldots,8\}$.
Indeed, assume there is a connected component $\calV\subseteq\{1,\ldots,n\}$ which has not this form, and let $u\in\{0,1\}^n$ be defined by, for all $j\in\{1,\ldots,n\}$, $u^{(j)}=1$ if and only if $j\in\calV$.
Let $n_1=\lvert\calV\cap\{1,\ldots,5\}\rvert$ and $n_2=\lvert\calV\cap\{6,\ldots,8\}\rvert$.
By assumption on $\calV$, $n_1,n_2\geq1$ and $n_1+n_2\leq7$.
Now, let $u'=U^\top u = [(u'^{(1:2)})^\top,(u'^{(3:8)})^\top]^\top$, where $u'^{(1:2)}\in\Re^2$ and $u'^{(3:8)}\in\Re^6$.
It holds that $u'^{(1:2)}=\big[\frac{n_1}{\sqrt{5}},\frac{n_2}{\sqrt{3}}\big]$. 
Fix $q\in\calQ$.
If $\calV$ is a connected component of $\calG_q$, it holds by Proposition \ref{pro-components-invariant-subspace} that $A_qu=u$, so that $A^{(21)}_qu'^{(1:2)}+A^{(22)}_qu'^{(3:8)} = u'^{(3:8)}$.
By the above, it follows that
\[
\lVert u'^{(1:2)}\rVert\geq\frac{1-\sqrt{1+\epsilon^2}\gamma}{\epsilon\gamma}\lVert u'^{(3:8)}\rVert\geq3.62\lVert u'^{(3:8)}\rVert.
\]
This implies that $\lVert u'^{(1:2)}\rVert\geq(1+(1/3.62)^2)^{-1/2}\lVert u'\rVert\geq0.96\lVert u'\rVert$, so that $\frac{n_1^2}5+\frac{n_2^2}3\geq0.96^2(n_1+n_2)$.
The maximum of $(\frac{n_1^2}5+\frac{n_2^2}3)/(n_1+n_2)$ is reached for $n_1=5$ and $n_2=2$, and equals $0.905$.
The latter is smaller than $0.96^2=0.9216$.
Hence, it follows that $u$ cannot satisfy $Au=u$, thus we can assert, with confidence level $99\%$, that $\calV$ is not a connected component of $\calG_q$.

\subsection{Opinion dynamics with antagonistic interactions}

We consider a problem of opinion dynamics, where many agents (e.g., the population of a country) are divided into four groups (e.g., based on their political affinity), and these agents exchange opinions about topics.
The interactions between the agents influence their opinion, and the relationships can be either friendly (in which case interaction of $A$ with $B$ increases the opinion of $A$ with the opinion of $B$) or antagonistic (in which case interaction of $A$ with $B$ increases the opinion of $A$ with the opposite of the opinion of $B$) \cite{altafini2013consensus}.
More precisely, for this example, the interaction patterns between the different groups can take three different values $(Q=3)$, represented in Figure \ref{fig-opinion-graphs}, and the value of the opinion of each group $j\in\{1,\ldots,4\}$ is updated as follows:
\begin{equation}\label{eq-opinion-rule}
\xi^{(j)}(t+1) = \frac1{1+n_j^{\sigma(t)}}\Big(\xi^{(j)}(t) + \sum_{j'\neq j} \sign^{\sigma(t)}_{j,j'}\xi^{(j')}(t)\Big),
\end{equation}
where $n_j^q$ is the number of groups with which $j$ interacts in the pattern $q$, $\sign^q_{j,j'}=1$ (resp.\ $-1$) if there is a friendly (resp.\ antagonistic) interaction of $j$ with $j'$, and $\sign^q_{j,j'}=0$ if there is no interaction.

\begin{figure}
\centering
\includegraphics[width=0.9\linewidth]{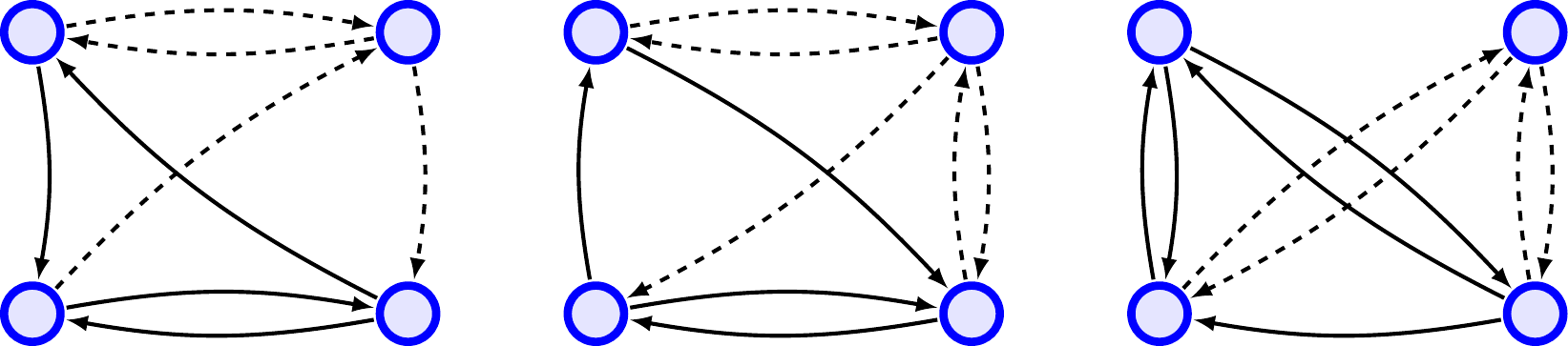}
\caption{Interaction patterns.
A solid (dashed) arrow from $A$ to $B$ indicates an interaction of $A$ with $B$ with friendly (antagonistic) relationship.}
\label{fig-opinion-graphs}
\end{figure}

In our setting, we do not know what are the interaction patterns; we only know an upper bound on $Q$.
But, we have access to the opinion vector (e.g., via polls or by counting the number of \emph{likes} on social networks) at different time instants, which allows us to collect sample trajectories of the system.
We are interested in deciding whether there is a \emph{stable} opinion vector, that is, whether there is a normalized vector $u\in\Re^4$, such that $A_qu=u$ for all $q\in\calQ$ (where $A_q$ is given by \eqref{eq-opinion-rule} with the interaction pattern $q$).
Our results do not allow us to ensure the existence of such a stable opinion vector; however, we are able to decide what will be this stable vector if it exists.%
\footnote{Once a potential stable vector is identified, one can apply further data-driven analysis to get a high confidence that this vector is indeed a stable one; however, for the sake of brievety and simplicity, we restrict here to the identification of potential stable vectors.}

To identify such a potential stable vector, we apply the data-driven quadratic Lyapunov framework described in Section \ref{sec-problem-setting}.
First, we compute the optimal solution of \eqref{eq-optim-black} with a relatively small sample set of one-step trajectories $\{(x_i,y_i)\}_{i=1}^{N_\smalltxt}$, $N_\smalltxt=2000$, of the system.
The associated optimal matrix satisfies $\Ker(P^\star_\smalltxt)=\spann\{u\}$, where $u=[1,-1,1,1]^\top$.
Hence, we suspect that $u$ is a stable vector of the system.
To verify this hypothesis with high confidence, we solve \eqref{eq-optim-black} with a larger sample set of one-step trajectories ($N_\largetxt=8\,142\,000$) and fixing $P=U^{(2)}(U^{(2)})^\top$ where the columns of $U^{(2)}\in\Re^{n\times(n-1)}$ are an orthonormal basis of $\{u\}^\perp$.
This time, we also try to minimize $\gamma$ (see also Remark~\ref{rem-constraints-P}).
This provides $\gamma^\star=0.334$.
By fixing the confidence level to $\beta=0.01$, it follows that $\epsilonbar(1)=5.67\cdot10^{-7}$.%
\footnote{We use $\epsilonbar(1)$ because $P$ is fixed, so that $d=0$; see Remark \ref{rem-constraints-P}.}
By Theorem \ref{thm-main-result}, it then follows that the matrices of the system admit a decomposition \eqref{eq-near-triangularizable}, with $\epsilon=0.02$, $\gamma=0.334$, $\Lambda=I_{n-1}$ and $U=[u,U^{(2)}]$.

From this decomposition, we deduce that if the system has a stable vector, then it must be close to $u$.
Indeed, assume there is a normalized stable vector $v\in\Re^4$ and let $v'=U^\top v=[v'^{(1)},(v'^{(2:3)})^\top]^\top$, where $v'^{(1)}\in\Re$ and $v'^{(2:3)}\in\Re^3$.
Fix $q\in\calQ$.
Since $v$ is a stable vector, it holds that $A_qv=v$, so that $A^{(21)}_qv'^{(1)}+A^{(22)}_qv'^{(2:3)} = v'^{(2:3)}$.
By the above, it follows that
\[
\lvert v'^{(1)}\rvert\geq\frac{1-\sqrt{1+\epsilon^2}\gamma}{\epsilon\gamma}\lVert v'^{(2:3)}\rVert\geq99\lVert v'^{(2:3)}\rVert.
\]
This implies that $\lvert v'^{(1)}\rvert\geq(1+(1/99)^2)^{-1/2}\geq0.9999$.
Hence, we finally deduce, with confidence level $99\%$, that any normalized stable vector $v$ for the system, if it exists, must be close to $u$, in the sense that $\lVert u-\pm v\rVert\leq0.0001$.

\section{Conclusions}\label{sec-conclusions}

We introduced a quadratic Lyapunov framework for data-driven identification of potential invariant subspaces of black-box switched linear systems.
This framework allows us to identify a potential invariant subspace for the system without knowing any mathematical model, by computing a quadratic Lyapunov function from a data-driven optimization program.
We then leverage results from scenario optimization, quasi-convex optimization, and geometric analysis, to come up with probabilistic guarantees on the identification accuracy.
We demonstrated the applicability of our framework on problems of consensus and opinion dynamics, for which the existence of an invariant subspace bears useful information about the system, thereby allowing us to study the dynamics of these systems in a data-driven way.

For further work, we plan to further investigate the possibility of adding constraints in the optimization program in order to leverage prior information about the system or the sampling (for instance, in the case of adaptive sampling), and also to fight the curse of dimensionality.
Another angle of attack to fight the curse of dimensionality is to refine the result asserting the homogeneity of a set from its probability measure, or to relax the condition of homogeneity while keeping the property of a bound on the identification accuracy.
Finally, we plan to provide other data-driven analysis results for the identification of invariant subspaces, for instance, regarding the minimal growth rate of the trajectories on the identified subspace, thereby allowing us to provide guarantees on the existence of an invariant subspace without relying on prior knowledge or assumptions as in the presented applications.


\ifextended

\section*{APPENDIX}

\subsection{Proof of Theorem \ref{thm-near-triangularizable}}\label{ssec-app-proof-thm-near-trianularizable}

Let $q\in\calQ$, $P=P^\star_\Omega$ and $U=[U^{(1)},U^{(2)}]$, where $U^{(1)}\in\Re^{n\times r}$ and $U^{(2)}\in\Re^{n\times(n-r)}$.

\emph{Step 1:} First, we show that $\lVert\Lambda A_q^{(21)}\rVert\leq\epsilon\gamma$.
Therefore, let $x'\in\Re^r$ and denote $x=U^{(1)}x'$.
It holds that $x^\top\Ptilde x=(U^{(1)}x')^\top U\Sigmatilde^2U^\top(U^{(1)}x')=x'^\top x'=\lVert x'\rVert^2$ and $(A_qx)^\top P(A_qx)=(A_qU^{(1)}x')^\top U\Sigma U^\top(A_qU^{(1)}x')=(A_q^{(21)}x')^\top\Lambda^2(A_q^{(21)}x')=\lVert\Lambda A_q^{(21)}x'\rVert^2$.
Using the homogeneity assumption, let $\{v_j\}_{j=1}^m\subseteq\Re^n$ and $\{\alpha_j\}_{j=1}^m\subseteq[0,1]$ be such that $\sum_{j=1}^m\alpha_jv_j=0$, $\sum_{j=1}^m\alpha_j=1$ and for each $j\in\{1,\ldots,m\}$, $v_j\!^\top Pv_j\leq\epsilon^2x^\top \Ptilde x$ and $x+v_j\in\calX_q$.
We have that
{\allowdisplaybreaks\begin{align*}
(A_qx)^\top P(A_qx) &= (\sum_{j=1}^m\alpha_jA_q(x+v_j))^\top P \\[-10pt]
&\hspace{2.8cm}(\sum_{j=1}^m\alpha_jA_q(x+v_j)) \\
\text{($P$ is PSD)}\quad &\leq \sum_{j=1}^m\alpha_j(A_q(x+v_j))^\top P(A_q(x+v_j)) \\
\text{(definition of $\calX_q$)}\quad &\leq \sum_{j=1}^m\alpha_j\gamma^2(x+v_j)^\top P(x+v_j) \\
\text{($x\in\Ker(P)$)}\quad &= \sum_{j=1}^m\alpha_j\gamma^2v_j\!^\top Pv_j \\
\text{(definition of $v_j$)}\quad &\leq \sum_{j=1}^m\alpha_j\epsilon^2\gamma^2x^\top\Ptilde x = \epsilon^2\gamma^2\lVert x'\rVert^2.
\end{align*}}%
Hence, $\lVert\Lambda A_q^{(21)}x'\rVert\leq\epsilon\gamma\lVert x'\rVert$.
Since $x'\in\Re^r$ was arbitrary, this holds for every $x'\in\Re^r$.
Hence, by definition of the spectral norm, $\lVert\Lambda A_q^{(21)}\rVert\leq\epsilon\gamma$, concluding Step 1.

\emph{Step 2:} Secondly, we show that $\lVert\Lambda A_q^{(22)}\Lambda^{-1}\rVert\leq\sqrt{1+\epsilon^2}\gamma$.
Therefore, let $x'\in\Re^{n-r}$ and denote $x=U^{(2)}x'$.
It holds that $x^\top Px=\lVert\Lambda x'\rVert^2$ and $(A_qx)^\top P(A_qx)=\lVert\Lambda A_q^{(22)}x'\rVert^2$.
Using the homogeneity assumption, let $\{v_j\}_{j=1}^m\subseteq\Re^n$ and $\{\alpha_j\}_{j=1}^m\subseteq[0,1]$ be such that $\sum_{j=1}^m\alpha_jv_j=0$, $\sum_{j=1}^m\alpha_j=1$ and for each $j\in\{1,\ldots,m\}$, $v_j\!^\top Px=0$, $v_j\!^\top Pv_j\leq\epsilon^2x^\top \Ptilde x$ and $x+v_j\in\calX_q$.
By using the same reasoning as in Step 1, we have that
{\allowdisplaybreaks\begin{align*}
(A_qx)^\top P(A_qx) &\leq \sum_{j=1}^m\alpha_j\gamma^2(x+v_j)^\top P(x+v_j) \\
\text{($v_j\!^\top Px=0$)}\quad &= \sum_{j=1}^m\alpha_j\gamma^2 (x^\top Px + v_j\!^\top Pv_j) \\
\text{(definition of $v_j$)}\quad &\leq \sum_{j=1}^m\alpha_j(1+\epsilon^2)\gamma^2x^\top\Ptilde x \\
&= (1+\epsilon^2)\gamma^2\lVert x'\rVert^2.
\end{align*}}%
Hence, $\lVert\Lambda A_q^{(22)}x'\Lambda^{-1}\rVert\leq\sqrt{1+\epsilon^2}\gamma\lVert x'\rVert$.
Since $x'\in\Re^{n-r}$ was arbitrary, this holds for every $x'\in\Re^{n-r}$.
Hence, by definition of the spectral norm, $\lVert\Lambda A_q^{(22)}\Lambda^{-1}\rVert\leq\sqrt{1+\epsilon^2}\gamma$, concluding the proof.

\subsection{Proof of Theorem \ref{thm-chance-constrained-probabilistic-bound-optim-black}}\label{ssec-app-proof-thm-chance-constrained-probabilistic-bound-optim-black}

Let $\Prob$ be as in Theorem~\ref{thm-chance-constrained-probabilistic-bound-optim-black}.
We will need the following key result from chance-constrained optimization.

\begin{theorem}[{\cite[Theorem 1]{garatti2021therisk}}]\label{thm-chance-constrained-probabilistic-bound}
Let $\beta\in(0,1)$ and $N\in\Ne$.
Let $\epsilonbar:\{0,\ldots,N\}\to[0,1]$ be as in Theorem~\ref{thm-chance-constrained-probabilistic-bound-optim-black}.
It holds that
\begin{equation}\label{eq-bound-risk}
\Prob^N(\{\psi\in(\Re^n\times\calQ)^N : \Prob(\Psi_\psi)<1-\epsilonbar(s^*(\psi))\})\leq\beta,
\end{equation}
where for all $\psi=((x_i,q_i))_{i=1}^N$, $\Omega_\psi=\{(x_i,A_{q_i}x_i)\}_{i=1}^N$ and $\Psi_\psi=\Psi_{\Omega_\psi}$ is as in \eqref{eq-omegabar}, and $s^*(\psi)$ is the smallest length of a subsequence $\varphi$ of $\psi$ such that $P^\star_{\Omega_{\varphi}}=P^\star_{\Omega_\psi}$.
\end{theorem}

As mentioned in \cite[p.\ 609]{garatti2021therisk}, computing $s^*(\psi)$ for a sequence $\psi\in(\Re^n\times\calQ)^N$ can be difficult as it amounts to look at all possible subsequences $\varphi$ of $\psi$, whose number grows combinatorially with the length of $\psi$.
However, in the specific case considered in this paper, we can obtain an upper bound on $s^*(\psi)$ that depends only on the dimension of System \eqref{eq-switched-system} and is in general much smaller than $N$.
The reason is that problem \eqref{eq-optim-black} is a convex optimization problem.

\begin{theorem}[{\cite[Lemma~2.10]{calafiore2010random}}]\label{thm-basis-cardinality}
Let $N\in\Ne$ and $\psi\in(\Re^n\times\calQ)^N$.
Let $s^*(\psi)$ be as in Theorem \ref{thm-chance-constrained-probabilistic-bound}.
It holds that $s^*(\psi)\leq d+1$, where $d=\frac{n(n+1)}2-1$ is the dimension of the variable $P$ in problem \eqref{eq-optim-black}.
\end{theorem}

The proof of Theorem \ref{thm-chance-constrained-probabilistic-bound-optim-black} follows from Theorems \ref{thm-chance-constrained-probabilistic-bound} and \ref{thm-basis-cardinality}.

\subsection{Proof of Theorem~\ref{thm-probability-homogeneity}}\label{ssec-app-proof-thm-probability-homogeneity}


To prove Theorem~\ref{thm-probability-homogeneity}, we will need the following definition and lemma.

For a set $S\subseteq\SSb^{n-1}$, $x\in\SSb^{n-1}$ and $s\in(0,1)$, we define the \emph{$s$-slice of $S$ along $x$} as
\[
\Pi_{x,s}(S) = \{v\in\{x\}^\perp:sx+\sqrt{1-s^2}v\in S\}.
\]
It is readily seen that for any $v\in\Pi_{x,s}(S)$, $\lVert v\rVert=1$, so that $\Pi_{x,s}(S)\subseteq\{v\}^\perp\cap\SSb^{n-1}\cong\SSb^{n-2}$.
The following lemma states that if the measure of $S$ is close to $1$, then for most $s\in(-1,1)$, the measure of $\Pi_{x,s}(S)$ must be close to $1$ as well.

\begin{lemma}\label{lem-measure-s-slice-sphere}
Let $S\subseteq\SSb^{n-1}$, $x\in\SSb^{n-1}$ and $\eta\in(0,1)$.
Assume that $\mu^{n-1}(S)\geq1-\eta$.
Let $\theta\in(\eta,1)$, and let
\[
\sbar=\sqrt{1-\Itilde_{\frac{n-1}2,\frac12}(\eta/\theta)},
\]
where $\Itilde_{a,b}$ is as in Theorem~\ref{thm-probability-homogeneity}.
Then, there is $s\in(-1,-\sbar]\cup[\sbar,1)$ such that $\mu^{n-2}(\Pi_{x,s}(S))\geq1-\theta$.
\end{lemma}


\begin{proof}
By Theorem \ref{thm-spherical-measure-product} in Appendix~\ref{ssec-app-spherical-measure}, it holds that if for all $s\in(-1,-\sbar]\cup[\sbar,1)$, $\mu^{n-2}(\SSb^{n-2}\setminus(\Pi_{x,s}(S)))>\theta$, then $\mu^{n-1}(\SSb^{n-1}\setminus S)>\frac{2\int_{\sbar}^1 (1-s^2)^{(n-3)/2}\,\diff s}{\int_0^1 t^{a-1}(1-t)^{b-1}\,\diff t}\theta$.
Using the change of variable $t=1-s^2$, we get that $\mu^{n-1}(\SSb^{n-1}\setminus S)>\frac{\int_0^{1-\sbar^2} t^{a-1}(1-t)^{b-1}\,\diff t}{\int_0^1 t^{a-1}(1-t)^{b-1}\,\diff t}\theta$, a contradiction with $\mu^{n-1}(S)\geq1-\eta$ and the definition of $\sbar$.
\end{proof}

We are now able to complete the proof of Theorem~\ref{thm-probability-homogeneity}.
Therefore, fix $q\in\calQ$.
First, note that $\mu^{n-1}(\SSb^{n-1}\setminus\calX_q)\leq Q\epsilonhat$, because if $\mu^{n-1}(\SSb^{n-1}\setminus\calX_q)>Q\epsilonhat$, then by definition of $\Prob_\circ$, $\Prob_\circ((\Re^n\times\calQ)\setminus\Psi_\Omega)>\epsilonhat$, which is a contradiction.
Decompose $U=[U^{(1)},U^{(2)}]$ where $U^{(1)}\in\Re^{n\times r}$ and $U^{(2)}\in\Re^{n\times(n-r)}$, and define $V=U\Sigmatilde^{-1}$ and $\Sigmahat=\diag(0_1,\ldots,0_r,1_1,\ldots,1_{n-r})$.
Let $\Ahat_q=V^{-1}A_qV$ and
\[
\calXhat_q=\{\xhat\in\Re^n:(\Ahat_q\xhat)^\top\Sigmahat(\Ahat_q\xhat)\leq\gamma^2 \xhat^\top\Sigmahat\xhat\}.
\]
It holds that $\calXhat_q=V^{-1}\calX_q$.
Thus, by Lemma \ref{lem-spherical-measure-linear-transform} in Appendix~\ref{ssec-app-spherical-measure}, $\mu^{n-1}(\SSb^{n-1}\setminus\calXhat_q)\leq\eta$, so that $\mu^{n-1}(\calXhat_q\cap\SSb^{n-1})\geq1-\eta$.
Fix $\xhat\in\SSb^{n-1}$.
Let $s\in(-1,\sbar]\cup[\sbar,1)$ be such that $\mu^{n-2}(\Pi_{\xhat,s}(\calXhat_q\cap\SSb^{n-1}))\geq\frac12$ (Lemma \ref{lem-measure-s-slice-sphere}).
Denote $\calW=\Pi_{\xhat,s}(\calXhat_q\cap\SSb^{n-1})$.
It holds that $0\in\conv(\calW)$ and for each $w\in\calW$, $s\xhat+\sqrt{1-s^2}w_j\in\calXhat_q\cap\SSb^{n-1}$.

Fix $\alpha\geq0$, and let $x=\alpha sV\xhat$.
It holds that $x^\top\Ptilde x=\alpha^2s^2\geq\alpha^2\sbar^2$.
Finally, define $\calV=\alpha\sqrt{1-s^2}V\calW$.
It holds that $0\in\conv(\calV)$, and for each $v\in\calV$, $v^\top P^\star_\Omega v\leq\alpha^2(1-s^2)\leq\alpha^2(1-\sbar^2)$.
Hence, for each $v\in\calV$, $v^\top P^\star_\Omega v\leq\epsilon^2x^\top\Ptilde x$ and $x+v\in\calX_q$.
Since $\xhat\in\SSb^{n-1}$ and $\alpha\geq0$ were arbitrary, it follows that $\calX_q$ is $(\epsilon,P^\star_\Omega,\Ptilde)$-homogeneous, concluding the proof of the theorem.

\subsection{Spherical measure on $\SSb^{d-1}$}\label{ssec-app-spherical-measure}

For any $d\in\Ne_{>0}$, let $\SSb^{d-1}$ be the unit Euclidean sphere in $\Re^d$, i.e., $\SSb^{d-1}$ is the boundary of the unit Euclidean ball $\BBb^d$ in $\Re^d$, and let $\mubar^d$ be the Lebesgue measure on $\Re^d$.
For a Lebesgue measurable set $S\subseteq\SSb^{d-1}$, its \emph{uniform spherical measure} is defined by $\mu^{d-1}(S)=\frac1{V_d}\mubar^d(\{tx : x\in S,\,t\in[0,1]\})$, where $V_d=\mubar^d(\BBb^d)$.

\begin{theorem}\label{thm-spherical-measure-product}
Let $S\subseteq\SSb^{n-1}$, and for each $s\in(-1,1)$, let $\Pi_s(S)=\{x'\in\SSb^{n-2}:[s,\sqrt{1-s^2}x'^\top]^\top\in S\}$.
It holds that
\[
\mu^{n-1}(S)=\frac1{B(\frac{n-1}2,\frac12)}\int_{-1}^1 (1-s^2)^{(n-3)/2}\mu^{n-2}(\Pi_s(S))\,\diff s,
\]
where $B(a,b)=\int_0^1 t^{a-1}(1-t)^{b-1}\,\diff t$ is the \emph{beta} function.
\end{theorem}

\begin{proof}
We will use the following result from \cite[Lemma 2.4.7]{willem2013functional}: for any $\alpha_0>0$ and $\calX\subseteq\Re^d$,
\[
\frac{\diff}{\diff\alpha}\mubar^d(\calX\cap\alpha\BBb^d)\Big\vert_{\alpha=\alpha_0} = \frac{d}{V_d}\alpha_0^{d-1}\mu^{d-1}\Big(\frac1{\alpha_0}\calX\cap\SSb^{d-1}\Big).
\]
Thus, letting $S^+=\{tx : x\in S,\,t\geq0\}$, we get that
\begin{equation}\label{eq-spherical-measure-derivative}
\mu^{n-1}(S) = \frac{V_n}{n}\frac{\diff}{\diff\alpha}\mubar^n(S^+\cap\alpha\BBb^n)\Big\vert_{\alpha=1}.
\end{equation}
Now, for each $s\in(-1,1)$, let $[S^+]_s=\{x'\in\Re^{n-1}:[s,x']\in S^+\}$.
By the product measure of $\mubar^n$, it holds that
\[
\mubar^n(S^+\cap\alpha\BBb^n) = \int_{-\alpha}^\alpha\mubar^{n-1}\big([S^+]_s\cap\sqrt{\alpha^2-s^2}\BBb^{n-1}\big)\,\diff s
\]
Hence, injecting in \eqref{eq-spherical-measure-derivative}, we get that
\begin{align*}
\mu^{n-1}(S) &= \frac{V_n}{n} \int_{-1}^1 \frac{\diff}{\diff\alpha}\mubar^{n-1}\big([S^+]_s\cap\sqrt{\alpha^2-s^2}\BBb^{n-1}\big)\Big\vert_{\alpha=1}\,\diff s \\
&= \frac{V_n}{n} \int_{-1}^1 \frac{n-1}{V_{n-1}}\frac{(1-s^2)^{(n-2)/2}}{(1-s^2)^{-1/2}}\mu^{n-2}(\Pi_s(S)) \,\diff s.
\end{align*}
The proof is completed by noting that, due to the normalization property $\mu^{n-1}(\SSb^{n-1})=1$, $\frac{V_{n-1}}{V_n}\frac{n}{n-1}$ is equal to $\int_{-1}^1 (1-s^2)^{(n-3)/2}\,\diff s$, which is shown to be equal to $B(\frac{n-1}2,\frac12)$ by using the change of variable $t=1-s^2$.
\end{proof}



\begin{lemma}[{\cite[Theorem 15]{kenanian2019data}}]\label{lem-spherical-measure-linear-transform}
Let $S\subseteq\SSb^{n-1}$, and let $A\in\Re^\nxn$ be invertible.
Let $\sigmabar_1\geq\ldots\geq\sigmabar_n>0$ be the singular values of $A$.
It holds that $\mu^{n-1}(\{Ax/\lVert Ax\rVert:x\in S\})\leq\frac{\prod_{j=1}^n\sigmabar_j^{1/n}}{\sigmabar_n}\mu^{n-1}(S)$.
\end{lemma}

\fi

\bibliographystyle{IEEEtran} 
\bibliography{IEEEabrv,myrefs}


\end{document}